\documentclass[12pt]{amsart}
\usepackage[cp1250]{inputenc}
\usepackage[T1]{fontenc}
\usepackage{amscd}

\theoremstyle{plain}
\newtheorem{thm}{Theorem}[section]
\newtheorem{lem}[thm]{Lemma}
\newtheorem{prop}[thm]{Proposition}
\newtheorem{cor}[thm]{Corollary}

\theoremstyle{definition}
\newtheorem{dff}[thm]{Definition}
\newtheorem{rem}[thm]{Remark}

\numberwithin{equation}{section}
\def\C{\mathcal{C}}

\def\D{\mathcal{D}}

\def\F{\mathcal{F}}
\def\H{\mathcal{H}}

\def\U{\mathcal{U}}

\def\L{\mathcal{L}}
\def\P{\mathcal P}
\def\Q{\mathcal Q}

\def\phi{\varphi}
\def\p{\partial}
\def\pp{\phi}

\def\rz{\mathbb{R}}
\def\R{\rz}

\def\N{\mathbb{N}}

\def\wyr#1{\textit{#1}}

\def\s{\subset}

\def\t{\times}
\def\r{\rightarrow}

\def\ld{,\ldots,}
\def\pp{\partial}

\DeclareMathOperator{\diff}{Diff}
 \DeclareMathOperator{\gau}{Gau}

 \DeclareMathOperator{\cl}{cl}

  \DeclareMathOperator{\conj}{conj}
 \DeclareMathOperator{\graph}{graph}
 \DeclareMathOperator{\frag}{frag}
 
\DeclareMathOperator{\id}{id} 
\DeclareMathOperator{\intt}{Int} \DeclareMathOperator{\supp}{supp}

\keywords{Perfect group, locally continuously perfect, locally
smoothly perfect, uniformly perfect, manifold, homeomorphism,
diffeomorphism, conjugation-invariant norm, group of
homeomorphisms, fragmentation, deformation }

\subjclass{57S05, 22A05, 22E65}

\thanks{Partially supported by the Polish Ministry of Science and Higher Education and the
AGH grant n. 11.420.04}
\address{Faculty of Applied Mathematics, AGH University of Science and
\linebreak Technology, al. Mickiewicza 30, 30-059 Krak\'ow,
Poland} \email{tomasz@uci.agh.edu.pl}
\date{December 15, 2010}

\title{Locally continuously perfect   groups of homeomorphisms}
\author{Tomasz Rybicki}

\begin{document}

\maketitle

\begin{abstract}
The notion of a locally continuously perfect  group is introduced
and studied.  This notion generalizes locally smoothly perfect
groups introduced by Haller and Teichmann. Next, we prove that the
path connected identity component of the group of all
homeomorphisms of a manifold is locally continuously perfect.  The
case of equivariant homeomorphism group and other examples are
also considered.
\end{abstract}

\section{Introduction}
Recall that a group $G$ is \emph{perfect} if it is equal to its
own commutator subgroup $[G,G]$. This means that for any $g\in G$
there exist $r\in\N$ and $h_i,\bar h_i\in G$, $i=1\ld r$, such
that
\begin{equation} g=[h_1,\bar h_1]\ldots[h_r,\bar
h_r].\end{equation} When we consider the category of topological
groups a fundamental question arises whether $h_i,\bar h_i$ can be
chosen to be continuously dependent on $g$. More precisely, we
introduce the following notion.
\begin{dff}
A topological group $G$ will be called \emph{locally continuously
perfect } if there exist $r\in\N$, a neighborhood $U$ of $e$ in
$G$ and
 continuous mappings $S_i:U\r G$, $\bar S_i:U\r G$, $i=1\ld r$,
such that \begin{equation}g=[S_1(g),\bar
S_1(g)]\ldots[S_{r}(g),\bar S_{r}(g)]\end{equation} for every
$g\in U$. Moreover, we assume that $S_i(e)=e$ for all $i$. The
smallest $r$ as above will be denoted by $r_G$.
\end{dff}
Clearly any connected locally continuously perfect group is
perfect.

An analogous notion of a \emph{locally smoothly perfect group} in
the category of (possibly infinite dimensional) Lie groups have
been studied in the  paper by Haller and Teichmann \cite{ha-te},
where, for the first time, the problem of smooth dependence of
$h_i,\bar h_i$ on $g$ in (1.1) was put forward.  The main purpose
of the present paper is to show that the property of being locally
continuously perfect is even more common for homeomorphism groups
of manifolds than its smooth counterpart in \cite{ha-te} for
diffeomorphism groups of manifolds. In both cases very deep (but
completely different from each other) facts are exploited: our
main result is based on deformations in the spaces of imbeddings
of manifolds (Edwards and Kirby \cite{ed-ki}), while
 Haller and Teichmann used a simplicity theorem of Herman (\cite{Her73}) on
the diffeomorphism group of a torus and the small denominator
theory (or the KAM theory) in its background.

Throughout $M$ is a topological metrizable manifold, possibly with
boundary, and  $\H(M)$  denotes the path connected identity
component  of the group of all  compactly supported homeomorphisms
of $M$ endowed with the graph topology (\cite{KM}) (or the
majorant topology \cite{ed-ki}). By a \emph{ball} in $M$ we will
mean rel. compact, open ball imbedded with its closure in $M$.
Similarly we define a half-ball if $M$ has boundary. For $M$
compact, let $d_M$ is the smallest integer such that
$M=\bigcup_{i=1}^{d_M}B_i$ where $B_i$ is a ball or half-ball for
each $i$.
\begin{thm} If $M$ is compact, then  $\H(M)$ is a locally continuously
perfect group (even more, it satisfies Def. 2.1 below) with
$r_{\H(M)}\leq d_M$. In particular, $\H(M)$ is perfect and simple
(if $M$ is connected).

\end{thm}

The fact that $\H(M)$ is perfect is an immediate consequence of
\cite{Mat71} and \cite{ed-ki}, Corollary 1.3. A special case was
already proved by Fisher \cite{fis}. Note that if we drop the
compactness assumption, $\H(M)$ is also perfect in view of an
argument based on Theorem 5.1 (also Theorem 5.1 in \cite{ed-ki}).
The group $\H(M)$ is simple as well, see e.g. \cite{li}.

\begin{thm} Let $M$ be an open  manifold such that $M=\intt(\bar M)$, where $\bar M$
is a compact manifold with boundary. Then $\H(M)$ is a  locally
continuously perfect group (and fulfils Def. 2.1). In particular,
$\H(M)$ is perfect and simple. Furthermore, $r_{\H(M)}\leq
d_{M}+2$. Here $d_M$ stands for the smallest integer such that
$M=P\cup \bigcup_{i=1}^{d_M}B_i$ where $B_i$ is an open ball  for
each $i$ and $P$ is a collar neighborhood of the boundary.

\end{thm}

 It is doubtful whether $\H(M)$ is locally continuously
perfect without the assumption of Theorem 1.3. Observe that
Theorems 1.2 and 1.3 are also true for isotopies (Corollary
3.3).

In the next three sections we present miscellaneous notions,
examples, facts and problems related to locally continuously
perfect groups.  The proofs of Theorems 1.2 and 1.3, making use of
subtle and difficult techniques of Edwards and Kirby in
\cite{ed-ki}, are presented in section 5.   The case of
$G$-equivariant homeomorphisms is investigated in section 6.

\section{Relative notions and basic lemma}
In order to describe the structure of homeomorphism  groups of
manifolds it will be useful to strengthen slightly Def.1.1.
\begin{dff}
 A topological group $G$ is \emph{locally continuously perfect (in a stronger sense)}
if there are $r\in\N$, a neighborhood $U$ of $e\in G$, elements
$h_1\ld h_r\in G$ and continuous mappings $S_i:U\r G$ with
$S_i(e)=e$, $i=1\ld r$, satisfying \begin{equation}
g=[S_1(g),h_1]\ldots[S_r(g),h_r]\end{equation} for all $g\in U$.
\end{dff}
Let us  "globalize" the notion of local continuous perfectness.

\begin{dff}
 A group $G$ is called \emph{uniformly perfect} (see, e.g., \cite{tsu2}),
 if $G$ is perfect and
there is $r\in\N$ such that any $g\in G$ can be expressed by (1.1)
for some $h_i,\bar h_i\in G$, $i=1\ld r$. Next, a topological
group $G$ is said to be \emph{continuously  perfect} if there
exist $r\in\N$ and continuous mappings $S_i:G\r G$, $\bar S_i:G\r
G$, $i=1\ld r$, satisfying the equality (1.2) for all $g\in G$.
\end{dff}
Of course, every continuously perfect group is uniformly perfect.

In early 1970's  Thurston proved that  the identity component of
the group of compactly supported diffeomorphism of class
$C^{\infty}$ of a manifold $M$ is perfect and simple (see
\cite{thu74}, \cite{Ban97}). The proof is based on Herman's
theorem \cite{Her73}.  Next, Thurston's theorem was extended to
groups of $C^r$-diffeomorphisms  where $r\neq\dim(M)+1$
(\cite{Mat74}), and to classical diffeomorphism groups
(\cite{Ban97},
 \cite{ry5}).
Recently, the problem of uniform perfectness of diffeomorphism
groups has  been studied in \cite{bip}, \cite{tsu2} and
\cite{ry6}. In contrast to the problem of perfectness and
simplicity, the obtained results depend essentially on the
topology of the underlying manifold.  However, in most cases (with
some exceptions presented below)
  difficult open problems arise whether the groups in question
 satisfy Def. 1.1, 2.1, or 2.2, or whether they are locally smoothly perfect.


The following type of fragmentations is important when studying
 groups of homeomorphisms.
\begin{dff} Let  $\U$ be an open  covering of $M$.
 A subgroup $G\s \H(M)$ is  \emph{locally continuously factorizable with respect to $\U$}
  if  for any finite subcovering $(U_i)_{i=1}^d$ of $\U$, there
 exist a neighborhood $\P$ of $\id\in G$ and continuous mappings $\sigma_i:\P\r G$, $i=1\ld d$,
  such that   for all $f\in
 \P$ one has \begin{equation*}f=\sigma_1(f)\ldots \sigma_d(f),\quad \supp(\sigma_i(f))\s U_i, \forall i.
 \end{equation*}
 \end{dff}

 Given a subset $S\s M$, by
$\H_S(M)$  we denote the path connected identity component of the
subgroup of all elements of $\H(M)$ with compact support contained
in $S$.

Using the Alexander trick, we have that $\H(\R^n)$ coincides with
the group of all compactly supported homeomorphisms of $\R^n$. In
fact, if $\supp(g)$ is compact, we define an isotopy $g_{t} :
\mathbb{R}^{n} \rightarrow \mathbb{R}^{n}$, $t \in I$, from the
identity to $g$, by

\begin{equation*}
g_{t}(x)= \left\{
\begin{array}{lcl}
tg\left( \frac{1}{t}x \right)& for & t>0\\
x&for&t=0.
\end{array} \right.
\end{equation*}

In particular, for every ball $B$ in $M$ the group $\H_B(M)$
consists of all homeomorphisms compactly supported in $B$.

The following fact, with a straightforward  proof, plays a basic
role in studies on homeomorphism groups.
\begin{lem}\cite{Mat71}(Basic lemma)
Let $B\s M$ be a ball and $U\s M$ be an open subset such that
$\cl(B)\s U$. Then there are $\phi\in\H_U(M)$ and a continuous
mapping $S:\H_B(M)\r \H_U(M)$ such that $h=[S(h),\phi]$ for all
$h\in\H_B(M)$.

\end{lem}
\begin{proof} First choose a larger ball $B'$ such that $\cl(B)\s
B'\s\cl(B')\s U$. Next, fix $p\in\pp B'$ and set $B_0=B$. There
exists a sequence of balls $(B_k)_{k=1}^{\infty}$ such that
$\cl(B_k)\s B'$ for all $k$, the family $(B_k)_{k=0}^{\infty}$ is
pairwise disjoint, locally finite in $B'$, and $B_k\r p$ when
$k\r\infty$. Choose a homeomorphism $\phi\in\H_U(M)$ such that
$\varphi(B_{k-1})=B_k$ for $k=1,2,\ldots$. Here we use the fact
that $\H_U(M)$ acts transitively on the family of balls in $B'$
(c.f. \cite{hir}).

Now we define a continuous homomorphism $S:\H_B(M)\r\H_U(M)$ by
the formula
$$ S(h)=\phi^kh\phi^{-k}\quad\hbox{on}\,B_k,\ k=0,1,\ldots$$
and $S(h)=\id$ outside $\bigcup_{k=0}^{\infty}B_k$. It is clear
that $h=[S(h),\phi]$, as required.
\end{proof}

\begin{rem} (1) Perhaps for the first time the above reasoning appeared in Mather's paper
\cite{Mat71}. Actually Mather proved also the acyclicity of
$\H(\R^n)$. Obviously, \cite{Mat71} and Lemma 2.4 are no longer
true for $C^1$ homeomorphisms. However, Tsuboi brilliantly
improved this reasoning and adapted it for $C^r$-diffeomorphisms
with small $r$, see \cite{tsu89}.

(2) It is likely that the basic lemma is no longer true for
$\H(M)$ instead of $\H_B(M)$ and $\H_U(M)$. In fact, consider
$\H(\R^n)$ and a weaker Def.1.1. If one tried to repeat the proof
of Lemma 2.4 then one would have continuous maps $S, \bar
S:\H(\R^n)\supset U\r\H(\R^n)$, where $\bar S(h)$ would play a
role of the shift homeomorphism $\phi$ for $h\in\H(\R^n)$. But
then $\bar S(h)$ depends somehow on the support of $h$. On the
other hand, there are arbitrarily close to id elements of
$\H(\R^n)$ with arbitrarily large support. This would spoil the
continuity of $\bar S$.
\end{rem}

Given a foliated manifold $(M,\F)$, a mapping $h:M\r M$ is
\emph{leaf preserving} if $h(L)\s L$ for all $L\in\F$. Let
$\H(M,\F)$ be the path connected identity component of the group
of all leaf preserving homeomorphisms of $(M,\F)$.
\begin{cor} Let $\F_k=\{\R^k\t\{pt\}\}$, $k=1\ld n-1$, be the
product foliation of $\R^n$. If $B=I\t \R^{n-k}$ and $U=J\t
\R^{n-k}$, where $I, J\s\R^k$  are open intervals such that
$\cl(I)\s J$, then there exist $\phi\in\H_U(\R^n,\F_k)$ and a
continuous mapping $S:\H_B(\R^n,\F_k)\r \H_U(\R^n,\F_k)$ such that
$h=[S(h),\phi]$ for all $h\in\H_B(\R^n,\F_k)$.

\end{cor}

\begin{proof}
Consider $\H_I(\R^k)$ and $\H_J(\R^k)$, and repeat the
construction of $\phi$ and $S(h)$ from the proof of Lemma 2.4.
Then multiply everything by $\id_{\R^{n-k}}$.
\end{proof}

\begin{cor} Assume that either
\begin{enumerate}\item M is a manifold with boundary and $B, U\s
M$ such that $B$ is a half-ball, $U$ is open with $\cl(B)\s U$; or
\item $M=N\t\R$, where $N$ is a manifold, and $B=N\t I$, $U=N\t J$
where $I, J\s\R$ are open intervals with $\cl(I)\s
J$.\end{enumerate} Then the assertion of Lemma 2.4 holds.
\end{cor}

The proof is analogous to the above.

\section{Examples}
First examples are provided by Theorems 1.2 and 1.3, and by
Theorem 6.2 below.

{\bf 1.} For a smooth manifold $M$ let $\D^r(M)$  stand for the
subgroup of all elements of $\diff^r(M)$  that can be joined to
the identity by a compactly supported isotopy in
 $\diff^r(M)$, where $r=1\ld\infty$.

 Let $G$ be a possibly infinite dimensional Lie group  which is simultaneously a topological group.
  If $G$ is also
locally smoothly perfect (see Def. 1 in Haller and Teichmann
\cite{ha-te}) then obviously $G$ is  locally continuously perfect
(even satisfies Def. 2.1).

In particular, the following groups are locally smoothly perfect
(and a fortiori locally continuously perfect).
\begin{enumerate}
\item Any finite dimensional perfect Lie group $G$; we then have
$r_G\leq\dim G$ (\cite{ha-te}). \item Any real semisimple Lie
group $G$; then $r_G=2$ (\cite{ha-te}). \item $\D^{\infty}(\mathbb
T^n)$ for the torus $\mathbb T^n$ with $r_{\D^{\infty}(\mathbb
T^n)}\leq 3$ (\cite{Her73}). \item Let $M$ be a closed smooth
manifold being the total space of $k$ locally trivial bundles with
fiber $\mathbb T^{n(i)}$, $i=1\ld k$, such that the corresponding
vertical distributions span $TM$. Then $\D^{\infty}(M)$ is locally
smoothly perfect. In particular, the assumption is satisfied for
odd dimensional spheres and for any compact Lie group $G$, and one
has $r_{\D^{\infty}(\mathbb S^3)}\leq 18$ and
$r_{\D^{\infty}(G)}\leq 3(\dim G)^2$ (see \cite{ha-te}).
\end{enumerate}

\begin{rem}
To obtain the theorem mentioned in (4) Haller and Teichmann used a
fruitful method of decomposing diffeomorphisms into fiber
preserving ones. This method enabled an application of the deep
Herman's theorem stating that $\D^{\infty}(\mathbb T^n)$ is not
only perfect and simple but also locally smoothly perfect. Notice
that this method was already considered in the "generic" case of
$M=\R^n$ in \cite{ry2} to study the (still open) problem of the
perfectness of $\D^{n+1}(M^n)$ by  means of the possible
perfectness property of the group of leaf preserving
$C^r$-diffeomorphisms with $r$ large. For a hypothetical  proof of
such a result one would apply Mather's proof of the simplicity of
$\D^r(M^n)$, $r\neq n+1$, from \cite{Mat74}. See also \cite{Mat74}
III, \cite{ry1}, \cite{LR} for the problem of the perfectness of
leaf preserving diffeomorphism groups.

Notice as well that recently Tsuboi in \cite{tsu3} used a similar
method in the proofs of
 perfectness theorems for  groups of real-analytic
 diffeomorphisms in absence of the fragmentation property.

\end{rem}

It seems likely that the groups $\D^r(\R^n)$ with small $r$
(depending on $n$) would be continuously perfect. By using his own
method Tsuboi in \cite{tsu89} generalized Mather's method from
\cite{Mat71}, reproved the simplicity theorem for
$C^r$-diffeomorphisms with $1\leq r\leq n$ (originally proved by
Mather \cite{Mat74}, II), and showed the vanishing of lower order
homologies of the groups $\D^r(\R^n)$. However, a possible
analysis of a very technical proof in \cite{tsu89} is beyond the
scope of the present paper. We may also ask whether $\D^r(\R^n)$
is (locally) $C^r$-smoothly perfect.
\medskip

{\bf 2.} It is very likely that theorems analogous to Theorems 1.2
and 1.3 can be obtained for the groups of Lipschitz homeomorphisms
on Lipschitz manifolds. See Theorem 2.2 and other results in Abe
and Fukui \cite{AF01}.
\medskip

{\bf 3.} Now we consider permanence properties of locally
continuously perfect groups. These properties provide further
examples of such groups.

 Let $H\s G$ be a subgroup and $G$ be  locally continuously perfect
  with continuous mappings $S_i:U\r G$, $\bar S_i:U\r G$,
$i=1\ld r$, satisfying (1.2). If $S_i(U\cap H)\s H$ and $\bar
S_i(U\cap H)\s H$ for all $i$ then $H$ is also  continuously
perfect. Corollary 2.6 illustrates this situation.

If $G$ and $H$ are locally continuously perfect groups then so is
its product $G\t H$.

For a compact manifold $M$ and a topological group $G$, let
$\C(M,G)$ stand for the group of continuous maps $M\r G$ with the
pointwise multiplication and the compact-open topology. Note that
$\C(M,G)$ can be viewed as an analogue of the current group (c.f.
\cite{KM}).

\begin{prop} If $G$ is locally continuously perfect then so is $\C(M,G)$ and $r_{\C(M,G)}=r_G$.
\end{prop}
\begin{proof} Let $S_i:U\r G$, $\bar S_i:U\r G$,
$i=1\ld r$, be as in Def. 1.1. Set $\U=\{f\in\C(M,G): f(M)\s U\}$
and define continuous maps $S_i^{\C}:\U\r \C(M,G)$, $\bar
S_i^{\C}:\U\r \C(M,G)$, $i=1\ld r$, by the formulae
$S_i^{\C}(f)(x)=S_i(f(x))$, where $f\in\C(M,G)$, $x\in M$, and
similarly for $\bar S_i^{\C}$. It follows that
$$ \prod_i[S_i^{\C}(f),\bar S_i^{\C}(f)](x)=\prod_i[S_i(f(x)),\bar
S_i(f(x))]=f(x)$$ for all $x\in M$. Thus
$f=\prod_i[S_i^{\C}(f),\bar S_i^{\C}(f)]$, as required. Observe
that for all $x\in M$ and $f\in\C(M,G)$, if $f(x)=e$ then
$S^{\C}_i(f)(x)=e$ for all $i$.
\end{proof}

 For a topological group $G$ we denote by $\P G=\{f:I\r G:\,
 f(0)=e\}$ the path group.
 \begin{cor}
 Theorems 1.2 and 1.3 hold for $\P\H(M)$. In
 other words, these theorems are true for isotopies.
 \end{cor}

Let $G$ be a compact Lie group. Given a principal $G$-bundle
$p:M\r B_M$ the \emph{gauge group} $\gau(M)$ is the group of all
$G$-equivariant mappings of $M$ over $\id_{B_M}$. That is,
$\gau(M)$ is the space of $G$-equivariant mappings
$\C(M,(G,\conj))^G$. It follows that $\gau(M)$ identifies with
$\C(B_M\leftarrow M[G,\conj])$, the space of sections of the
associated bundle $M[G,\conj]$. Consequently, any $f\in\gau(M)$ in
a trivialization of $p$ over $B=B_i\s B_M$ identifies with a
mapping $f^{(i)}:B_i\r G$ such that $f(x)=x.f^{(i)}(p(x))$.

\begin{prop}
Let $G$ be a compact Lie group and let $p:M\r B_M$ be a principal
$G$-bundle with $B_M$ compact. Then $\gau(M)$ is locally
continuously perfect provided $G$ is so, and $r_{\gau(M)}=d_M r_G$
where $d_M$ is as in Theorem 1.2.
\end{prop}
\begin{proof}
Let $(B_i)_{i=1}^d$ be a covering of $M$ by balls. Choose another
covering  by balls $(B_i')_{i=1}^d$ with $\cl(B_i')\s B_i$ for all
$i$. We identify $\L(G)$, the Lie algebra of $G$, with $\rz^q$,
$q=\dim G$, by means of a basis $(X_1\ld X_q)$  of $\L(G)$. Let
$\Phi:\rz^q\supset V\r U\s G$ be a chart given by $\Phi(t_1\ld t
_q)=(\exp\,t_1X_1)\ldots(\exp\,t_qX_q)$. Suppose that
$h\in\gau(M)$ is so small that the image of $ h^{(1)}$ is in $U$.
 We let $\tilde h^{(1)}=\Phi^{-1}\circ\ h^{(1)}$.
 By using bump functions for
$(B_1',B_1)$ we modify $\tilde h^{(1)}$ and compose the resulting
map with $\Phi$. Consequently, we get $g_1\in\gau(M)$ such that
$\supp(g_1)\s p^{-1}(B_1)$ and $g^{(1)}_1=h^{(1)}$ on $B_1'$.
Moreover, $g_1$ depends continuously on $h$. Now take
$f_1=g_1^{-1}h$. Then $f_1\in\gau(M_{1})$, where $M_1=M\setminus
B_1'$ is a compact manifold with boundary endowed with the
coverings $(B_i\cap M_1)_{i=1}^d$ and $(B_i'\cap M_1)_{i=1}^d$.
Note that $f_1=\id$ on $\p B'_1$.  Taking possibly smaller $h$ we
may continue the procedure. Finally, we get a neighborhood $\U$ of
$e\in\gau(M)$ such that for all $h\in\U$ we get a uniquely
determined decomposition $h=h_1\ldots h_d$ with $\supp(h_i)\s
p^{-1}(B_i)$ and $h_i$ depending continuously on $h$ for all $i$.

Thus, possibly shrinking $\U$, in view of Proposition 3.2 the
claim follows.
\end{proof}

\section{Remarks on conjugation-invariant norms}

The notion of the conjugation-invariant norm is a basic tool in
studies on the structure of groups.  Let $G$ be a group. A
\wyr{conjugation-invariant norm} (or \emph{norm} for short) on $G$
is a function $\nu:G\r[0,\infty)$ which satisfies the following
conditions. For any $g,h\in G$ \begin{enumerate} \item $\nu(g)>0$
if and only if $g\neq e$; \item $\nu(g^{-1})=\nu(g)$; \item
$\nu(gh)\leq\nu(g)+\nu(h)$; \item $\nu(hgh^{-1})=\nu(g)$.
\end{enumerate}
Recall that a group is called \emph{ bounded} if it is bounded
with respect to any bi-invariant metric. It is easily seen that
$G$ is bounded if and only if any conjugation-invariant norm on
$G$ is bounded.

Let us introduce the following norm.
\begin{dff}
 Let $G$ be a connected topological group and let $U$ be a neighborhood  $e\in G$.
\begin{enumerate}\item By $\tilde U$ we denote the "saturation" of $U$ w. r. t.
$\conj_g$ for $g\in G$ and
 the inversion $i$, that is $\tilde U=\bigcup_{g\in G}gUg^{-1}\cup gU^{-1}g^{-1}$. Then for $g\in
G$, $g\neq e$, by $\mu^{ U}(g)$ we denote the smallest $s\in\N$
such that $g=g_1\ldots g_s$ with $g_i\in \tilde U$ for  $i=1\ld
s$. It is easily seen that $\mu^{ U}$ is a conjugation-invariant
norm. \item  We say that $G$ is \emph{continuously decomposable
with respect to $U$} if there are $s\in\N$ and continuous mappings
$\varrho_i:G\r \tilde  U$, $i=1\ld s$, such that
$g=\varrho_1(g)\ldots\varrho_s(g)$ for all $g\in G$. In
particular,  $\mu^U(G)\leq s$.
\end{enumerate}
\end{dff}
Clearly if $U\s V$ then $\mu^V\leq\mu^U$.

It is straightforward that if $G$ is locally continuously perfect
and continuously decomposable w.r.t. $U$ as in Def. 1.1, then $G$
is continuously perfect. Likewise, if $G\s\H(M)$ satisfies Def.
2.3 and if $G$ is continuously decomposable w.r.t. $\P$ as in Def.
2.3, then $G$ is continuously factorizable. However, we do not
have any example of a  homeomorphism group being continuously
decomposable. Notice that, in view of Lemma 2.4, for any two balls
$B$ and $U$ in $M$ with $\cl(B)\s U$ the group $\H_B(M)$ is
continuously perfect "in the group $\H_U(M)$", but not in itself.

Recall  now two classical examples of conjugation-invariant norms.

For $g\in[G,G]$ by the \emph{commutator length } of $g$,
$\cl_G(g)$, we mean the smallest $r$ as in (1.1). Observe that the
commutator length $\cl_G$ is a  norm on $[G,G]$. In particular, if
$G$ is a perfect group then $\cl_G$ is a  norm on $G$.

A subgroup $G\s\H(M)$ is \emph{factorizable} if for every $g\in G$
there are $g_1\ld g_d\in G$ with $\supp(g_i)\s B_i$, where each
$B_i$ is a ball or a half-ball.  Clearly any connected $G$
satisfying Def. 2.3 with respect to a family of balls  is
factorizable. If $G$ is factorizable then we may introduce the
following  \emph{ fragmentation norm} $\frag_G$ on $G$. For $g\in
G$, $g\neq\id$, we define $\frag_G(g)$ to be the least integer
$d>0$ such that $g=g_1\ldots g_{d}$ with $\supp(g_i)\s B_i$ for
some ball or half-ball $B_i$.

 Let $\nu$ be a conjugation-invariant norm on a topological group
$G$. Then $G$ is called \emph{locally bounded with respect to
$\nu$} if there are $r\in\N$ and a symmetric neighborhood $U$
 of $e\in G$ (i.e. $U=U^{-1}$) such that $\nu(g)\leq
r$ for all $g\in U$. The following obvious fact can be applied to
$\cl_G$ or to $\frag_G$.

\begin{cor} Let $G$ be a subgroup of $\H(M)$. If $G$ is locally bounded w. r. t. $\nu$
and the norm $\mu^{U}$ (Def. 4.1) is bounded where $U$ is as
above, then $G$ is bounded w. r. t. $\nu$.

\end{cor}

\section{Proofs of Theorems 1.2 and 1.3}

The proofs  depend on the deformation properties for the spaces of
imbeddings  obtained by Edwards and Kirby in \cite{ed-ki}. See
also Siebenmann \cite{sie}. First let us recall some notions  and
the main theorem of \cite{ed-ki}. From now on $M$ is a metrizable
topological manifold and $I=[0,1]$. If $U$ is a subset of  $M$, a
\emph{proper imbedding} of $U$ into $M$ is an imbedding $h: U
\rightarrow M$ such that $h^{-1}(\partial M)=U \cap \partial M$.
An \emph{isotopy} of $U$ into $M$ is a family of imbeddings
$h_{t}: U \rightarrow M$, $t \in I$, such that the map $h: U
\times I \rightarrow M$ defined by $h(x,t)=h_{t}(x)$ is
continuous. An isotopy is \emph{proper} if each imbedding in it is
proper. Now let $C$ and $U$ be  subsets of $M$ with $C\subseteq
U$. By $I(U,C;M)$ we denote the space of proper imbeddings of $U$
into $M$
 which equal the identity on $C$, endowed with the compact-open topology.

Suppose $X$ is a space with subsets $A$ and $B$. A
\emph{deformation of A into B}
 is a continuous mapping $\varphi : A \times I\rightarrow X$ such that $\varphi|_{A\times 0}=\id_{A}$
  and $\varphi(A\times 1) \subseteq B$. If $\P$ is a subset of
  $I(U;M)$ and $\varphi:\P \times I\rightarrow I(U;M)$ is a deformation of $\P$,
  we may equivalently view $\varphi$ as a map $\varphi:\P\times I\times U\rightarrow M$
  such that for each $h\in \P$ and $t\in I$, the map $\varphi(h,t):U\rightarrow M$ is
  a proper imbedding.


If $W\subseteq U$, a deformation $\varphi: \P\times I\rightarrow
I(U;M)$ is \emph{modulo W} if $\varphi(h,t)|_{W}=h|_{W}$ for all
$h\in \P$ and $t\in I$.

 Suppose $\varphi: \P\times I \rightarrow
I(U;M)$ and $\psi: \Q\times I \rightarrow I(U;M)$ are deformations
of subsets of $I(U;M)$ and suppose that $\varphi(\P\times
1)\subseteq \Q$. Then the \emph{composition} of $\psi$ with
$\varphi$,  denoted by $\psi \star \varphi$, is the deformation
$\psi \star \varphi :\P\times I\rightarrow I(U;M)$ defined by
\begin{equation}
\psi \star\varphi(h,t)= \left\{
\begin{array}{lcl}
\varphi(h,2t)& for & t\in [0,1/2]\\
\psi(\varphi(h,1),2t-1)& for &t\in [1/2,1].
\end{array}
\right.
\end{equation}


The main result in \cite{ed-ki} is the following

\begin{thm}
Let $M$ be a topological manifold and let $U$ be a neighborhood in
$M$ of a compact subset $C$. For any neighborhood $\Q$ of the
inclusion $i:U\subset M$ in $I(U;M)$ there are a neighborhood $\P$
of $i\in I(U;M)$ and a deformation $\varphi:\P\t I\r \Q$ into
$I(U,C;M)$ which is modulo of the complement of a compact
neighborhood of $C$ in $U$ and such that $\varphi(i,t)=i$ for all
$t$. Moreover, if $D_i\subset V_i$, $i=1\ld q$, is a finite family
of closed subsets $D_i$ with their neighborhoods $V_i$, then
$\varphi$ can be chosen so that the restriction of $\varphi$ to
$(\P\cap I(U,U\cap V_i;M))\t I$ assumes its values in $I(U,U\cap
D_i;M)$ for each $i$.
\end{thm}

Now we wish to show that $\H(M)$ is locally continuously
factorizable (Def. 2.3) provided $M$ is compact.
\begin{prop} Let $M$ be compact and let $(U_i)_{i=1}^d$ be an open
cover of $M$. Then there exist $\P$, a neighborhood of the
identity in $\H(M)$, and continuous mappings $\sigma_i:\P\r\H(M)$,
$i=1\ld d$, such that $h=\sigma_1(h)\ldots \sigma_d(h)$ and
$\supp(\sigma_i(h))\s U_i$ for all $i$ and all $h\in\P$; that is
$\H(M)$ satisfies Def. 2.3.
\end{prop}
\begin{proof} (See also \cite{ed-ki}.)
First we have to shrink the cover $(U_i)_{i=1}^d$ $d$ times, that
is we choose an open $U_{i,j}$ for every $i=1\ld d$ and $j=0\ld d$
with $U_{i,0}=U_i$ such that $\bigcup_{i=1}^dU_{i,j}=M$ for all
$j$ and such that $\cl(U_{i,j+1})\s U_{i,j}$ for all $i,j$.  We
make use of Theorem 5.1 $d$ times with $q=1$. Namely, for $i=1\ld
d$ we have a neighborhood $\P_i$ of the identity in
$I(M,\bigcup_{\alpha=1}^{i-1}U_{\alpha,i-1}; M)$ and a deformation
$\phi_i:\P_i\t I\r \H(M)$ which is modulo $M\setminus U_{i,0}$ and
which takes its values in
$I(M,\bigcup_{\alpha=1}^{i}\cl(U_{\alpha,i}); M)$ and such that
$\phi_i(\id, t)=\id$ for all $t$. Here we apply Theorem 5.1 with
$C=\cl(U_{i,i})$, $U=U_{i,0}$,
$D_1=\bigcup_{\alpha=1}^{i-1}\cl(U_{\alpha,i})$ and
$V_1=\bigcup_{\alpha=1}^{i-1}U_{\alpha,i-1}$. Taking a
neighborhood $\P$ of id small enough, we have that
$\phi_d\star\cdots\star\phi_1$ restricted to $\P\t I$ is well
defined. For every $h\in\P$ we set $h_0=h$ and
$h_i=\phi_i\star\cdots\star\phi_1(h,1)$, $i=1\ld d$. It follows
that $h_d=\id$ and $h=\prod_{i=1}^dh_ih_{i-1}^{-1}$. It suffices
to define $\sigma_i:\P\r\H(M)$ by $\sigma_i(h)=h_ih_{i-1}^{-1}$
for all $i$.\end{proof}


\medskip
\emph{Proof of Theorem 1.2}.  Choose any  finite cover
$(U_i)_{i=1}^d$ of $M$ by balls and half-balls. Next, fix another
cover of $M$ by balls  and half-balls $(B_i)_{i=1}^d$ with
$\cl(B_i)\s U_i$ for all $i$. Then apply Proposition 5.2 to
$(B_i)_{i=1}^d$, Lemma 2.4 and Corollary 2.7(1) to each couple
$(B_i, U_i)$.\quad $\square$

\medskip

Recall the notion of the  graph topology. Let $X$ and $Y$ be
Hausdorff spaces and let $\C(X,Y)$ be the space of all continuous
mappings $X\r Y$. For $f\in\C(X,Y)$ by $\graph_f:X\r X\t Y$ we
denote the graph mapping. The \emph{graph topology}
 on $\C(X,Y)$ is given by the basis of
all sets of the form $\{f\in\C(X,Y):\, \graph_f(X)\s U\}$, where
$U$ runs over all open sets in $X\t Y$. The graph topology is
Hausdorff since it is finer than the compact-open topology. If $X$
is paracompact and $(Y,d)$ is a metric space then for
$f\in\C(X,Y)$ one has a basis of neighborhoods of the form
$\{g\in\C(X,Y):\, d(f(x), g(x))<\varepsilon(x),\ \forall x\in
X\}$, where $\varepsilon$ runs over all positive continuous
functions on $X$.

\medskip

\emph{Proof of Theorem 1.3}. In view of the assumption $M$ is the
interior of a compact, connected manifold $\bar M$  with non-empty
(not necessarily connected) boundary $\p$. Then $\p$ admits a
collar neighborhood, that is an open subset $P$ of $M$, where
$P=\p\t(0,1)$. Here $\p\t[0,1]$ is imbedded in $\bar M$, and
$\p\t\{1\}$ is identified with $\p$.

 Take a finite family of balls $(B_i)_{i=1}^d$ in $M$  and a collar neighborhood $P$
 of $\p$ such
that  $M=\bigcup_{i}B_{i}\cup P$. We wish to check that $\H(M)$
fulfils Def. 2.3 for $\{B_1\ld B_d,P\}$. We define balls $B_{i,j}$
for every $i=1\ld d$ and $j=0\ld d$ with $U_{i,0}=U_i$ such that
$\bigcup_{i=1}^dU_{i,j}\cup P=M$ for all $j$ and such that
$\cl(U_{i,j+1})\s U_{i,j}$ for all $i,j$. Now proceeding like in
the proof of Prop.5.2 there exist a neighborhood $\P$ of
$\id\in\H(M)$ and continuous mappings $\sigma_i:\P\r\H(M)$, where
$i=0\ld d$, such that for all $h\in \P$ we have
$$h=\sigma_0(h)\sigma_1(h)\ldots\sigma_d(h),\quad \supp(\sigma_0(h))\s
P,\quad\supp(\sigma_i(h))\s B_i$$ for $i=1\ld d$. Fix a sequence
of reals from $(0,1)$
$$ \tilde a_1<\bar a_1<a_1<b_1<\bar b_1<\tilde b_1<\tilde
a_2<\cdots<\tilde a_k<\bar a_k<a_k<b_k<\bar b_k<\tilde
b_k<\cdots$$ tending to 1. For $j=1,2\ld $ let
$C_j=\p\t[a_j,b_j]$, $V_j=\p\t(\bar a_j,\bar b_j)$ and
$U_j=\p\t(\tilde a_j,\tilde b_j)$. In view of Theorem 5.1, for
every $j$ there exist  a neighborhood $\P_j$ of the inclusion
$i_j:C_j\s U_j$ in $I(U_j;M)$ and a deformation $\varphi_j:\P_j\t
I\r I(U_j,C_j;M)$ which is modulo of $M\setminus V_j$ and such
that $\varphi(i_j,t)=i_j$ for all $t\in I$. Shrinking $\P$ if
necessary,  for any $h\in\P$ we may have that
$\sigma_0(h)|_{U_j}\in\P_j$ for all $j$. Put $U=\bigcup U_j$,
$V=\bigcup V_j$ and let $D=\bigcup \p\t I_j$, where $I_j$ is an
arbitrary open interval with $\cl(I_j)\s(a_j,b_j)$ for all $j$.
Therefore there are a neighborhood $\P$ of $\id\in\H(M)$ in the
graph topology and a continuous mapping $\sigma_0^1:\P\r\H(M)$
given by
$$\sigma_0^1(h)|_{U_j}=\phi_j(\sigma_0(h)|_{U_j}),\quad j=1,2\ld$$
and $\sigma_0^1(h)|_{M\setminus U}=\sigma_0(h)|_{M\setminus U}$.
It follows that $\sigma_0^1(h)=\sigma_0(h)$ on $M\setminus V$ and
that $\supp(\sigma_0^1(h))\s M\setminus \cl(D)$. Set
$\sigma_0^2=(\sigma_0^1)^{-1}\sigma_0$. Then
$\sigma_0^2:\P\r\H(M)$ is continuous, and
$\sigma_0=\sigma_0^1\sigma_0^2$ with $\supp(\sigma_0^2)\s U$. Thus
we get a decomposition
$h=\sigma_0^1(h)\sigma_0^2(h)\sigma_1(h)\ldots\sigma_d(h)$ for all
$h\in\P$. By applying Lemma 2.4 to $\sigma_i$ and Corollary 2.7(2)
to $\sigma_0^j$, the claim follows. $\square$

\section{The case of $G$-equivariant homeomorphisms}

 Let $G$ be a
compact  Lie group  acting on $M$. Let $\H_G(M)$ be the group of
all equivariant homeomorphisms of $M$ which are isotopic to the
identity through compactly supported equivariant isotopies.
Suppose now that $G$ acts freely on $M$. Then $M$ can be regarded
as the total space of a principal $G$-bundle $p:M\r  B_M=M/G$
(c.f. \cite{br}).

Let $\C_c(\rz^m)$ (resp. $\C_B(\rz^m)$) denote the space of
continuous maps $u:\R^m\r\R$ with compact support (resp. contained
in $B$). Consider the semi-direct product group
$\H(\rz^m)\t_{\tau}\C_c(\rz^m)$, where
 $\tau_h(u)=u\circ h^{-1}$ for $h\in
\H(\rz^m)$ and $u\in \C_c(\rz^m)$. Then we have
$$(h_1,u_1)\cdot(h_2,u_2)=(h_1\circ h_2, u_1\circ h_2^{-1}+u_2)$$
for all $h_1,h_2\in\H(\rz^m)$ and $u_1,u_2\in \C_c(\rz^m)$. For
$(h, u)\in\H(\rz^m)\t_{\tau}\C_c(\rz^m)$ we have $(h,
u)=(h,0)\cdot(\id,u)=(\id,u_1)\cdot(h,0)$, where $u_1= u\circ h$.
We may treat $h,u$ as elements of $\H(\rz^m)\t_{\tau}\C_c(\rz^m)$.

The main lemma in \cite{ry3} (Lemma 2.1), which has an elementary
but rather sophisticated proof, can be reformulated for our
purpose as follows.
\begin{lem} Let $B$ be a ball in $\R^m$. There are homeomorphisms $\phi^-, \phi^+,
\psi^-, \psi^+$ from $\H(\R^n)$, depending on $B$, and continuous
mappings $$v_1^-, v_1^+, v_2^-, v_2^+:\C_B(\R^m)\r\C_c(\R^m)$$
such that
$$u=[\phi^-,v_1^-(u)]^{-1}[\phi^+,v_1^+(u)]^{-1}[\psi^-,v_2^-(u)]
[\psi^+,v_2^+(u)]$$ for all $u\in\C_B(\R^m)$ in the
 semi-direct product group
$\H(\rz^m)\t_{\tau}\C_c(\rz^m)$.
\end{lem}

In view of Lemma 6.1 we have
\begin{thm}
If $B_M$ is compact then the group $\H_G(M)$ is locally
continuously perfect. Moreover, $r_{\H_G(M)}\leq (4\dim
G+1)d_{B_M}$.
\end{thm}

\begin{proof} Let $(B_i)_{i=1}^d$ be a covering by balls of $B_M$.
 Let $P:\H_G(M)\r\H(B_M)$ be the
homomorphism given by $P(h)(p( x))=p(h(x))$, where $x\in M$. Let
$h\in\U$, where $\U$ is a neighborhood of $\id\in\H_G(M)$. Then
for $\U$ small enough $P(h)$ can be decomposed as $P(h)=g_1\cdots
g_d$ such that $g_i\in\H_{B_i}(B_ M)$, $i=1\ld d$. Then each $g_i$
can be  lifted to $h_i\in\H_G(M)$, i.e. $P(h_i)=g_i$. Thus, due to
Theorem 1.2 it suffices to consider $f=h h_d^{-1}\cdots
h_1^{-1}\in\ker P=\gau(M)$.

Proceeding as in the proof of Prop. 3.4, we can write $f=f_1\cdots
f_d$ where $f_i\in\gau(M)$ and $\supp(f_i)\s p^{-1}(B_i)$.
Shrinking $\U$ we assume that
$f_i\in\H(\rz^m)\t_{\tau}\C_c(\rz^m,\rz^q)$ for all $i$, where
$q=\dim G$.
 We can extend the semi-direct product structure from
$\H(\rz^m)\t_{\tau}\C_c(\rz^m)$ to
$\H(\rz^m)\t_{\tau}\C_c(\rz^m,\rz^q)$, where $\C_c(\rz^m,\rz ^q)$
is the space of compactly supported $\rz^q$-valued functions, by
the formulae $(h,(v_1\ld v_q))=(\id,(v_1\ld v_q)\circ
h)\cdot(h,0)$ and $(\id,(v_1\ld v_q))=(\id,v_1)\cdots(\id,v_q)$.
In view of Lemma 6.1, each $(\id,v_i)$ is written as a product of
four commutators from $\H(\rz^m)\t_{\tau}\C_c(\rz^m,\rz^q)$ with
factors depending continuously on $f$. This completes the proof.
\end{proof}

\begin{cor} Let $M$ be a topological $G$-manifold with one orbit
type. Then $\H_G(M)$ is a locally continuously perfect group.
\end{cor}
\begin{proof}
Indeed, if $H$ is the isotropy group of a point of $M$ then
$M^H=\{x\in M: \, H\,\,\hbox{fixes}\,\,x\}$ is a free
$N^G(H)/H$-manifold, where $N^G(H)$ is the normalizer of $H$ in
$G$. Since $\H_G(M)$ is isomorphic and homeomorphic to
$\H_{N(H)/H}(M^H)$, the corollary  follows from Theorem 6.2.

To explain the relation $\H_G(M)\cong\H_{N(H)/H}(M^H)$, recall
basic facts on the $G$-spaces with one orbit type (see, Bredon
\cite{br}, section II, 5). Let $G$ a compact Lie group and let $X$
be a $T_{3\frac{1}{2}}$ $G$-space with one orbit type $G/H$ (that
is, all isotropy subgroups are conjugated to $H$). Set $N=N^G(H)$
and $X^H=\{x\in X: h.x=x\,,\forall h\in H\}$. Then we have the
homeomorphism $G\t_NX^H\ni[g,x]\mapsto g(x)\in X$. That is, the
total space of the bundle over $G/N$ with the standard fiber $X^H$
associated to the principal $N$-bundle $G\r G/N$ is $G$-equivalent
to $X$. In particular, the inclusion $X^H\subset X$ induces a
homeomorphism $X^H/N\cong X/G$.

Denote $K=N/H$. Given an arbitrary $G$-space $Y$, there is a
bijection $\kappa_{X,Y}$ between $G$-equivariant mappings $X\r Y$
and $K$-equivariant mappings $X^H\r Y^H$ such that
$\kappa_{X,Y}(f)=f|_{X^H}$.

Notice that $K$ acts freely on $X^H$ and the homeomorphism
$X^H/N\cong X/G$ induces the homeomorphism $X^H/K\cong X/G$. In
particular, we get the principal $K$-bundle $\pi_X:X^H\r X/G$,
where $\pi_X$ is the restriction to $X^H$ of the projection
$\pi:X\r X/G$.
\end{proof}

\end{document}